\documentclass{amsart}

\newtheorem{theo}{Theorem}

\newtheorem{coro}{Corollary}
\newtheorem{lemm}{Lemma}

\theoremstyle{remark}
\newtheorem{rema}{\bf Remark}

\begin{document}

\title{Automorphism groups of dessins d'enfants}
\author{Rub\'en A. Hidalgo}

\address{Departamento de Matem\'atica y Estad\'{\i}stica, Universidad de La Frontera.  Temuco, Chile}
\email{ruben.hidalgo@ufrontera.cl}

\thanks{Partially supported by Projects Fondecyt 1150003 and Anillo ACT 1415 PIA-CONICYT}
\subjclass[2010]{Primary 30F40, 11G32, 14H57}
\keywords{Riemann surfaces, dessins d'enfants, automorphisms}

\maketitle

%%%%%%%%%%%%%%%%%%
\begin{abstract}
Recently, Gareth Jones observed that every finite group $G$ can be realized as the group of automorphisms of some dessin d'enfant ${\mathcal D}$. In this paper, complementing  Gareth's result, we prove that for every possible action of  $G$ as a group of orientation-preserving homeomorphisms on a closed orientable surface of genus $g \geq 2$, there is a dessin d'enfant ${\mathcal D}$ admitting $G$ as its group of automorphisms and realizing the given topological action.
 In particular, this asserts that the strong symmetric genus of $G$ is also the minimum genus action for it to acts as the group of automorphisms of a dessin d'enfant of genus at least two.
\end{abstract}

%%%%%%%%%%%%%%%%%%
\section{Introduction}
In \cite{Hurwitz}, Hurwitz showed that every finite group $G$ can be realized as a group of conformal automorphisms of some closed Riemann surface of genus $g \geq 2$, and that the upper bound $|G| \leq 84(g-1)$ holds. This upper bound is attained if and only if $G$ can be generated by two elements $a,b$ such that $a^{2}=b^{3}=(ab)^{7}=1$. In \cite[Thm. 4]{Greenberg2}, Greenberg observed that the Riemann surface can be chosen to have $G$ as its full group of conformal automorphisms. In this paper we study the realization of $G$ as the group of automorphisms of some dessin d'enfant.

A {\it dessin d'enfant} (as introduced by Grothendieck in his Esquisse \cite{Gro}) is a pair ${\mathcal D}=(X,{\mathcal G})$, where ${\mathcal G}$ is a bipartite graph (its vertices are either black or white and adjacent ones have different colour) embedded in a closed orientable surface $X$ and providing a $2$-cell decomposition of it, that is, each of the connected components of $X\setminus {\mathcal G}$, its faces, is a topological disc. The degree of a face is the number of its boundary edges.  If all its black vertices (respectively, all its white vertices and all its faces) have the same degree, then ${\mathcal D}$ is called {\it uniform}. An automorphism of ${\mathcal D}$ is a graph automorphism of the bipartite graph ${\mathcal G}$ (that is, it sends white vertices to white vertices)
induced by an orientation-preserving self-homeomorphism of $X$. The group of automorphisms of ${\mathcal D}$ is denoted by ${\rm Aut}({\mathcal D})$. 
If  ${\rm Aut}({\mathcal D})$ acts transitively on the edges of ${\mathcal G}$, then ${\mathcal D}$ is called {\it regular}, in which case it is necessarily uniform (there are examples of non-regular uniform dessins d'enfants). Generalities on dessin d'enfants can be found, for instance, in \cite{GG, Gro, J-S, JW, LZ}.

It is well known fact that a finite group is isomorphic to ${\rm Aut}({\mathcal D})$, for some regular dessin d'enfant ${\mathcal D}$, if and only if it is $2$-generated. So, if a finite group cannot be generated by two elements, then it cannot be realized as the group of automorphisms of a regular dessin d'enfant, but it appears as the group of automorphisms of a non-regular one as it was observed by G. Jones.

\begin{theo}[G. Jones \cite{Jones}]\label{main0}
Every finite group is isomorphic to ${\rm Aut}({\mathcal D})$ for a suitable dessin d'enfant ${\mathcal D}$. 
\end{theo}

As a consequence of the uniformization theorem, a dessin d'enfant ${\mathcal D}=(X,{\mathcal G})$ defines a 
closed Riemann surface $S$ (of the same genus as $X$ and called a {\it Belyi curve}) and a non-constant meromorphic map
$\beta:S \to \widehat{\mathbb C}$ whose branch values are contained in the set $\{0,1,\infty\}$ (called a {\it Belyi map}) such that $\beta^{-1}([0,1])$ is isomorphic to ${\mathcal G}$ (the pair $(S,\beta)$ is called a {\it Belyi pair}). Conversely, every Belyi pair $(S,\beta)$ defines a dessin d'enfant ${\mathcal D}=(X,{\mathcal G})$, where $X$ is the topological surface underlying $S$ and ${\mathcal G}= \beta^{-1}([0,1])$. This provides an equivalence between the categories of dessins d'enfants and that of Belyi pairs.
In this equivalence, the group ${\rm Aut}({\mathcal D})$ corresponds to the group ${\rm Aut}(S,\beta)$ of automorphisms of the Belyi pair, consisting of those conformal automorphisms $\phi$ of $S$ satisfying that $\beta=\beta \circ \phi$. In this seeting, Theorem \ref{main0} can be written as follows.

\begin{theo}[G. Jones \cite{Jones}]
Every finite group is isomorphic to ${\rm Aut}(S,\beta)$ for a suitable Belyi pair $(S,\beta)$. 
\end{theo}

Let $G_{j}$, for $j=1,2$, be a finite group of orientation-preserving sel-homeomorphisms of a closed orientable surface $X_{j}$ such that $G_{1} \cong G_{2}$ (isomorphic groups). We say that the actions of $G_{1}$ and $G_{2}$ are {\it topologically equivalent} if there is an orientation-preserving homeomorphism $F:X_{1} \to X_{2}$ such that $F G_{1} F^{-1}=G_{2}$. A consequence of uniformization theorem is that every finite group of orientation-preserving self-homeomorphisms of a closed orientable surface $X$ is topologically equivalent to the action of an isomorphic group of conformal automorphisms of a closed Riemann surface (see also \cite{Kerckhoff}).

In this paper we observe the following extension to the above Jones' realization theorem.

\begin{theo}\label{main1}
Let $G$ be a finite group of conformal automorphisms of  a closed Riemann surface $S$ of genus $g \geq 2$.
Then there is a Belyi pair $(\widehat{S},\beta)$, where $\widehat{S}$ has genus $g$, $G \cong {\rm Aut}(\widehat{S},\beta)$ and the actions of ${\rm Aut}(\widehat{S},\beta)$ and $G$ are topologically equivalent. 
\end{theo}

In terms of dessins d'enfants, the above can be written as follows.

\begin{theo}\label{main1-1}
Let $G$ be a finite group of conformal automorphisms of  a closed Riemann surface $S$ of genus $g \geq 2$.
Then there is a dessin d'enfant ${\mathcal D}=(X,{\mathcal G})$, where $X$ has genus $g$, $G \cong {\rm Aut}({\mathcal D})$ and the actions of ${\rm Aut}({\mathcal D})$ and $G$ are topologically equivalent. 
\end{theo}

\begin{rema}\label{referee}
As noted to the author by the referee, if the Belyi pair $(S,\beta)$ (dessin d'enfant) of genus $g>0$ is uniform, then $(S,\eta \circ \beta)$, where $\eta(z)=16z(z-3/4)^{2}$, is a non-uniform one with ${\rm Aut}(S,\beta)={\rm Aut}(S, \eta \circ \beta)$.
This, in particular, asserts that every $2$-generated finite group can be realized as the full group of automorphisms of both a regular dessin d'enfant and of a non-uniform dessin d'enfant.
\end{rema}

As previously observed, every finite group $G$ can be realized as a group of conformal automorphisms of some closed Riemann surface of genus at least two. The minimal value $\mu(G)$ of such genera is called the {\it strong symmetric genus} of  $G$. In \cite{MZ}, May and Zimmermann observed that every $g \geq 2$ is the strong symmetric group of a suitable finite group. The value of $\mu(G)$ is known for some groups (${\rm PSL}_{2}(q)$, ${\rm SL}_{2}(q)$, sporadic simple groups, alternating and symmetric groups, etc).  Theorem \ref{main1-1} asserts that a minimal genus action of a finite group $G$ can be realized over a suitable dessin d'enfant.

\begin{coro}\label{main3}
If $G$ is a finite group, then there is a dessin d'enfant of genus $\mu(G)$  admitting it as its group of automorphisms.
\end{coro}

In the above, we have considered finite groups acting conformally on closed Riemann surfaces. In \cite{Wink}, Winkelmann proved that every countable group can be realized as the full group of conformal automorphisms of some Riemann surface.
A subtle adaption of the arguments in the proof of Theorem \ref{main1} permits to observe the following for the finitely generated groups.

\begin{theo}\label{main2}
Every finitely generated  group can be realized as the full group of conformal automorphisms of a non-constant meromorphic map $\beta:S \to \widehat{\mathbb C}$ branched over at most three values, where $S$ is some (non-compact if $G$ is not finite) Riemann surface. 
\end{theo}

\subsection*{Acknowledgments}
The author is very grateful to Gareth Jones for discussions about his preprint \cite{Jones}, which motivated this work, and  to Ernesto Girondo and Gabino Gonz\'alez-Diez for the many discussions about dessins d'enfants. Also, the author would like to thank the referee for the valuable corrections/suggestions and by pointing out Remark \ref{referee}.

%%%%%%%%%%%%%%%%
\section{Proof of theorem \ref{main1}}
Let us start with the following helpful observation.

\begin{lemm}\label{lema1}
For each pair of integers $\gamma,r \geq 0$, there is a non-uniform Belyi pair $(\widehat{R},\delta)$, where $\widehat{R}$ is closed Riemann surface of genus $\gamma$, $\delta$ has some prime degree and with $\delta^{-1}(1)$ of cardinality at least $r$.
\end{lemm}
\begin{proof}
Make a triangulation of a closed orientable surface $Y$ of genus $\gamma$ such that it contains at least $r$ edges. Such a triangulation defines a clean dessin d'enfant of some degree (black vertices are the vertices of the triangulation and white vertices are the ones in the middle of edges of the triangulation). Now, we add edges from a white vertex $v$, into the interior of one of the two faces containing $v$ in its boundary,  in order to get a prime number of edges in total. As there are white vertices of degree two and the white vertex $v$ will have degree at least three, the previous defines a non-uniform dessin d'enfant of genus $\gamma$ which is non-uniform. The Belyi pair induced by this dessin d'enfant is the required one.
\end{proof}

%%%%%%%%%%%%%%%%%%%%
\subsection{Proof of theorem \ref{main1}}
Let $G$ be some finite group acting as a group of conformal automorphisms of a closed Riemann $S$ of genus $g \geq 2$. The quotient orbifold $S/G$ consists of a closed Riemann surface $R$ of genus $\gamma \geq 0$ with some finite number of cone points, $p_{1},\ldots, p_{r}$, such that, for $r>0$, the cone point $p_{j}$ has cone order $m_{j} \geq 2$, and these numbers satisfy $2\gamma+r-2 > \sum_{j=1}^{r}m_{j}^{-1}$. 
As a consequence of the uniformization theorem, there are a Fuchsian group $K$ acting on the hyperbolic plane ${\mathbb H}^{2}$ with presentation
$$K=\langle A_{1},B_{1},\ldots, A_{\gamma},B_{\gamma}, C_{1},\ldots, C_{r}: \prod_{j=1}^{\gamma}[A_{j},B_{j}] \prod_{k=1}^{r} C_{k}=1=C_{1}^{m_{1}}=\cdots=C_{r}^{m_{r}}\rangle,$$
where $[A_{j},B_{j}]=A_{j}B_{j}A_{j}^{-1}B_{j}^{-1}$,
and a surjective homomorphism $\theta:K \to G$, whose kernel $\Gamma_{\theta}$ is torsion-free, such that $S={\mathbb H}^{2}/\Gamma_{\theta}$, $S/G={\mathbb H}^{2}/K$ and the regular branched cover $S \to S/G$ is induced by the inclusion $\Gamma_{\theta} \leq K$.

Let us choose a Belyi pair $(\widehat{R},\delta)$, where $\widehat{R}$ has genus $\gamma$, $\delta^{-1}(1)$ has cardinality at least $r$  and $\delta$ has prime degree (as in Lemma \ref{lema1}). Let us make a choice of points $q_{1},\ldots,q_{r} \in \delta^{-1}(1)$. As a consequence of quasiconformal deformation theory \cite{Nag}, we may find a Fuchsian group $\widehat{K}$ such that ${\mathbb H}^{2}/\widehat{K}$ is the orbifold whose underlying Riemann surrace is $\widehat{R}$ and whose cone points are $q_{1},\ldots,q_{r}$ such that $q_{j}$ has cone order $m_{j}$. The group $\widehat{K}$ is isomorphic to $K$ (as abstract groups). Then there is a quasiconformal homeomorphism $W:{\mathbb H}^{2} \to {\mathbb H}^{2}$ conjugating $K$ into $\widehat{K}$. In this case, $\widehat{S}={\mathbb H}^{2}/W \Gamma_{\theta} W^{-1}$ is a closed Riemann surface admitting the group $\widehat{G}=\widehat{K}/W \Gamma_{\theta} W^{-1} \cong G$ as a group of conformal automorphisms and such that the actions of $G$ on $S$ and that of $\widehat{G}$ on $\widehat{S}$ are topologically equivalent (in particular, $\widehat{S}/\widehat{G}$ is also topologically equivalent to $S/G$). 

We have obtained a Belyi pair $(\widehat{S},\beta=\delta \circ \eta)$, where $\eta:\widehat{S} \to \widehat{R}$ is a regular branched cover with deck group $\widehat{G} \cong G$.  In this case, the group $\widehat{G}$ is a (which might be proper) subgroup of ${\rm Aut}(\widehat{S},\beta)$ whose conformal action is topolologically equivalent to that of $G$. Next, we proceed to observe that $G \cong {\rm Aut}(\widehat{S},\beta)$. 

Let us assume $\widehat{G}$ is a proper subgroup of  ${\rm Aut}(\widehat{S},\beta=\delta \circ \eta)=\widetilde{G}$. It means that the branched cover $\delta:\widehat{R} \to \widehat{\mathbb C}$ factors through $\widehat{S}/\widetilde{G}$. As $\delta$ has prime degree and $\widehat{G} \neq \widetilde{G}$, it must be that $(\widehat{S},\beta)$ is a regular dessin d'enfant. But as $\eta$ is an unbranched cover and $(\widehat{R},\delta)$ is non-uniform, neither can be $(\widehat{S},\beta)$, a contradiction as a regular Belyi pairs or dessins d'enfants are uniform.

%%%%%%%%%%%%
%%%%%%%%%%%%%%%%%%
\section{Proof of Theorem \ref{main2}}
As $G$ is a finitely generated group, we may find a finite set $\{g_{1},\ldots,g_{s}\}$ of generators for it. As a consequence of Lemma \ref{lema1}, 
 there is a non-uniform Belyi pair $(R,\delta)$, where $R$ is closed Riemann surface of genus $s$ and $\delta$ has order some prime integer $p \geq 3$.  Let 
$\Gamma=\langle A_{1},B_{1},\ldots,A_{s},B_{s}: \prod_{j=1}^{s} [A_{j},B_{j}]=1\rangle$ be a Fuchsian group acting on the hyperbolic plane ${\mathbb H}^{2}$ so that ${\mathbb H}^{2}/\Gamma=R$, and  consider the surjective homomorphism
$\theta:\Gamma \to G: A_{j} \mapsto g_{j}, B_{j} \mapsto 1.$
If $\Gamma_{\theta}$ is the kernel of $\theta$, then $S={\mathbb H}^{2}/\Gamma_{\theta}$ is a Riemann surface (which happens to be compact of genus $|G|(s-1)+1$ if $G$ is finite). The inclusion $\Gamma_{\theta} \leq \Gamma$ induces a regular unbranched covering $\eta:S \to R$, with deck group $G$. The pair $(S,\beta=\delta \circ \eta)$ has $G$ as a subgroup of its group of automorphisms (this is a Belyi pair if $G$ is finite). Next, we may proceed similarly as in the previous case to check that $G$ is the full group of automorphisms of the pair $(S,\beta)$. That is, if $G$ is not the full group $\widehat{G}$ of automorphisms of $(S,\beta)$, then the branched cover $\delta:R \to \widehat{\mathbb C}$ must factors through $S/\widehat{G}$. As $\delta$ has prime degree and $G \neq \widehat{G}$, it must be that $\beta$ is a regular branched cover with deck group $\widehat{G}$. But as $\eta$ is an unbranched cover and $(R,\delta)$ is non-uniform, neither can be $(S,\beta)$, a contradiction (since every regular branched covering is uniform).

%%%%%%%%%%%%%%%%%%%%%
%%%%%%%%%%%%%%%%%%%%%

\end{document}